\documentclass[12pt]{amsart}
\usepackage{amsmath}
\usepackage{amsfonts}
\usepackage{amssymb}
\usepackage{graphicx}
\usepackage{amsthm,graphicx,color,yfonts}
\usepackage{pdfsync}
\usepackage{epstopdf}
\usepackage{marginnote}
\usepackage{esint}

\usepackage[colorlinks=true]{hyperref}
\hypersetup{linkcolor=red,citecolor=blue,filecolor=dullmagenta,urlcolor=darkblue} 

\numberwithin{equation}{section}

\theoremstyle{plain}
\newtheorem{theorem}{Theorem}[section]

\newtheorem{lemma}[theorem]{Lemma}
\newtheorem*{de-lemma}{Lemma}
\newtheorem{proposition}[theorem]{Proposition}

\theoremstyle{remark}

\theoremstyle{definition}

\newtheorem{remark}{Remark}

\DeclareMathOperator{\supp}{supp}

\newcommand{\dd}{\mathrm{d}}
\newcommand{\R}{\mathbb{R}}
\newcommand{\SF}{\mathbb{S}}
\newcommand{\n}{\textbf{\em n}}

\providecommand{\MR}{\relax\ifhmode\unskip\space\fi MR }

\providecommand{\href}[2]{#2}

\begin{document}

\title[A comparison principle]{A comparison principle for vector valued minimizers of semilinear elliptic energy, with application to dead cores}

\author{Panayotis Smyrnelis} \address[P.~ Smyrnelis]{Basque Center for Applied Mathematics, 48009 Bilbao, Spain}
\email[P. ~Smyrnelis]{psmyrnelis@bcamath.org}

\begin{abstract}

We establish a comparison principle providing accurate upper bounds for the modulus of vector valued minimizers of an energy functional, associated when the potential is smooth, to elliptic gradient systems. Our assumptions are very mild: we assume that the potential is lower semicontinuous, and satisfies a monotonicity condition in a neighbourhood of its minimum. As a consequence, we give a sufficient condition for the existence of dead core regions, where the minimizer is equal to one of the minima of the potential.\\

\textbf{MSC2020}: Primary 35B51; 35J50; Secondary 35B50.
\textbf{Keywords}: comparison principle, phase transition, dead core, vector valued minimizer, semilinear elliptic energy.

\end{abstract}

\maketitle

\section{Introduction}

The scope of this paper is to establish a general comparison principle providing accurate upper bounds for the modulus of \emph{vector valued minimizers} of the energy functional
\begin{equation}\label{ene}
E_\omega(v):=\int_{\omega} \Big[\frac{1}{2}|\nabla v(x)|^2+W(v(x))\Big]\dd x, \ v\in W^{1,2}(\omega;\R^m), \ \omega \subset \R^n, \ n,m\geq 1,
\end{equation}
where $W:\R^m\to [0,\infty)$ is a nonnegative, lower semicontinuous potential (cf. Theorem \ref{th1} below). Concerning the behaviour of $W$ in a neighbourhood of one of its zero, supposed to be located at the origin, we shall only make two basic monotonicity assumptions (cf. $\mathbf{H_3}$ below). Namely, that in a neighbourhood of $0$:
\begin{itemize}\label{hip}
 \item[($m1$)] $W_{\mathrm{rad}}(|u|)\leq W(u) $, where $W_{\mathrm{rad}}:[0,q]\to[0,\infty)$ is a nondecreasing, lower semicontinuous function\footnote{We shall see in Theorem \ref{th1} that the upper bound obtained for the modulus of the minimizer, only depends on the profile of the function $W_{\mathrm{rad}}$.},
\item[($m2$)] and $u \mapsto W(u)-W_{\mathrm{rad}}(|u|)$ is nondecreasing on the rays emanating from the origin.
\end{itemize}
Thus, our result applies to a large class of potentials, including for instance the interesting particular case of the characteristic function of $\R^m\setminus\{0\}$. Phase transition problems involving nonsmooth potentials are often considered in the literature. We mention in particular the work \cite{alt} on free boundaries; the density estimates obtained in \cite{caf} (resp. \cite{alikakos2}) in the scalar (resp. vector) case; the properties of minimal surfaces and minimizers studied in \cite{savin}; the heteroclinic orbit problem examined in \cite{smy}; the structure of minimizers described in \cite{dbr} in the one dimensional case $n=m=1$. Although the potential $W$ may be a very rough function, we recall that the minimizers of \eqref{ene} are continuous maps (cf. Lemma \ref{lem2}).

Comparison principles are useful in phase transition problems, to study the convergence of a solution to the minima of the potential. The most typical situation occurs (cf. \cite[Lemma 4.4]{book}) when $W:\R^m\to [0,\infty)$ is a smooth potential  such that   
\begin{equation}
W\geq 0, W(0)=0, \text{and }\nabla W(u)\cdot u \geq c |u|^2, \text{ holds for } |u|\leq q
\end{equation}
(i.e. the minimum $0$ is nondegenerate), and $u\in C^2(\overline{\Omega};\R^m)$ is a smooth solution to $\Delta u (x)=\nabla W(u(x))$ in $\Omega\subset\R^n$, such that $|u|\leq q$ holds in $\Omega$. Then, in view of the inequality
\begin{equation}
\Delta |u|^2(x) \geq 2 \nabla W(u(x))\cdot u(x)\geq 2c |u(x)|^2, \forall x\in \Omega,
\end{equation}
the maximum principle implies that
\begin{equation}
|u(x)|^2 \leq \Phi(x), \forall x\in \Omega,
\end{equation}
where $\Phi:\Omega\to \R$ is the solution of the problem\footnote{We refer again to \cite[Appendix A]{book} for the decay properties of $\Phi$.}
\begin{equation}
\Delta \Phi= 2c \Phi \text{ in } \Omega, \text{ and } \Phi= q^2 \text{ on }\partial\Omega.
\end{equation}

On the other hand, we would also like to recall a classical result (cf. \cite[Theorem 7.2]{pucci1}), on the existence of \emph{dead core solutions} in the scalar case.
Let $W\in C^1([0,q];\R)$ be a potential defined on the interval $[0,q]$, and assume also that
\begin{subequations}
\begin{equation}\label{ww1}
W \text{ is convex},
\end{equation}
\begin{equation}\label{ww2}
W(0)=W'(0)=0,\text{ and }  W'>0 \text{ on } (0,q],
\end{equation}
\begin{equation}\label{lon1}
\int_0^q\frac{\dd s}{\sqrt{W(s)}}<\infty.
\end{equation}
\end{subequations}
Then, in every ball $B_R:=\{x \in \R^n: |x|<R\}$, the equation
\begin{equation}\label{sca}
\Delta u(x) =W'(u),  x \in B_R,
\end{equation}
admits a nonnegative \emph{dead core solution}, that is, a solution of \eqref{sca} satisfying
\begin{subequations}
\begin{equation}
u\equiv 0 \text{ in an open set } \omega \text{ such that } \overline\omega\subset B_R,
\end{equation}
\begin{equation}
u>0 \text{ in } B_R\setminus \overline\omega.
\end{equation}
\end{subequations}
Actually, the condition \eqref{lon1} is necessary and sufficient for the existence of dead cores. Indeed, the conditions \eqref{ww1}, \eqref{ww2}, and 
\begin{equation}\label{lon2}
\int_0^q\frac{\dd s}{\sqrt{W(s)}}=\infty,
\end{equation}
ensure the validity of the \emph{strong maximum principle} (cf. \cite[Theorem 1.1]{pucci1}): a nonnegative solution $u$ of $\Delta u\leq W'(u)$ defined in a connected open set $\Omega\subset\R^n$, is either positive or identically zero on $\Omega$.
The sufficiency of \eqref{lon2} for the strong maximum principle to hold is due to V\'azquez \cite{vaz}, while necessity is due to Benilan et al. \cite{benilan}. We refer to \cite{pucci1,pucci2,pucci3} and the references therein, for general statements of maximum and comparison principles, as well as for the theory of dead core solutions.

As far as \emph{vector valued minimizers} $u$ of \eqref{ene} are concerned (with $W$ a nonnegative, lower semicontinuous potential satisfying the monotonicity assumptions $(m1)$-$(m2)$), we shall see in Theorem \ref{th2} below, that the condition
\begin{equation}\label{lon3}
\int_0^q\frac{\dd s}{\sqrt{W_{\mathrm{rad}}(s)}}<\infty,
\end{equation}
still implies the existence of \emph{dead core regions}, where the minimizers $u$ vanish\footnote{In particular, whenever the function $W_{\mathrm{rad}}$ is discontinuous at $0$, dead core regions appear.}. However, in our \emph{variational setting}, the convexity of $W$ considered in \eqref{ww1} for \emph{solutions} of \eqref{sca} is not anymore required. Theorem \ref{th2} follows from the bound obtained in Theorem \ref{th1}, and also provides a general computation of the distance of the dead core from the boundary of the domain. Estimates for the dead core have initially been established e.g. in \cite{aris,bandle,diaz,fried,sperb}, and we refer to \cite[Section 8.4.]{pucci3} for further explicit examples. Finally, we point out that in the scalar case ($m=1$), if instead of \eqref{lon3} we assume that   
\begin{equation}\label{lon4}
\int_0^q\frac{\dd s}{\sqrt{W(s)}}=\infty \text{ (with $W(u)\equiv W_{\mathrm{rad}}(|u|)$)} 
\end{equation}
holds, then the existence of dead cores is partially ruled out by the following variational version of the maximum principle (cf. Proposition \ref{prop2} below): a minimizer $u:\R^n\supset\Omega \to[-q,q]$ of \eqref{ene} that is positive at the boundary of a subdomain $\omega\subset\subset \Omega$, is also positive on $\overline{\omega}$.

\section{Main results}
Now, we shall state more precisely our assumptions and main results. Let us assume that $B_q\subset \R^m$ is the open ball of radius $q>0$ centered at the origin, and that $W:\overline{B_q}\to [0,\infty)$ is a potential such that
\begin{itemize}\label{hypw}
\item[$\mathbf{H_1}$]  $W\geq 0$ and $W(0)=0$,
\item[$\mathbf{H_2}$]  $W$ is lower semicontinuous and bounded on $\overline{B_q}$,
\item[$\mathbf{H_3}$]  $W(u)=W_{\mathrm{rad}}(|u|)+W_0(u)$, 
with $W_{\mathrm{rad}}:[0,q]\to[0,\infty)$ a nondecreasing, lower semicontinuous function, and $W_0:\overline{B_q}\to[0,\infty)$ a function such that $W_0(r\xi)\leq W_0(s\xi)$ holds for every $0\leq r\leq s\leq q$, and every unit vector $\xi\in\R^m$.
\end{itemize}
Our comparison principle applies to maps $u\in W^{1,2}_{\mathrm{loc}}(\Omega;\R^m)$ defined in an open set $\Omega\subset\R^m$, such that
\begin{equation}\label{qbound}
\|u\|_{L^\infty(\Omega;\R^m)}\leq q,
\end{equation}
and $u$ is a \emph{local minimizer} of the energy functional \eqref{ene}, 
for perturbations satisfying \eqref{qbound}. That is, for every bounded open set $\omega$ with Lipschitz boundary, such that $\overline \omega\subset\Omega$, and every perturbation $v=u+\xi$ such that $\xi \in W^{1,2}_0(\omega;\R^m)$ and 
$\|v\|_{L^\infty(\omega;\R^m)}\leq q$, we assume that
\begin{equation}\label{enemin}
E_\omega(u)\leq E_\omega(v).
\end{equation}
For instance, if $\Omega$ is a smooth domain, and if we extend $W$ on the whole space $\R^m$ by setting
\begin{equation*}
\tilde W(u)=\begin{cases}
W(u) &\text{when } |u|\leq q,\\
W(\frac{qu}{|u|}) &\text{when } |u|\geq q,
\end{cases}
\end{equation*}
one can check that assumptions \eqref{qbound} and \eqref{enemin} hold, for every minimizer $u$ of 
\begin{equation*}
\tilde E_\Omega(v):=\int_{\Omega} \Big[\frac{1}{2}|\nabla v(x)|^2+\tilde W(v(x))\Big]\dd x,
\end{equation*}
in the class of maps $v\in W^{1,2}(\Omega;\R^m)$ satifying the boundary condition
\begin{equation*}
v=\phi \text{ on } \partial\Omega,  \text{ with } \phi \in W^{1,2}(\Omega;\R^m), \text{ and }\|\phi\|_{L^\infty(\Omega;\R^m)}\leq q.
\end{equation*}
In the following Theorem \ref{th1}, we shall establish an upper bound for the modulus of the local minimizer $u$. Our comparison function $\overline{\Psi}_R$ is defined in 

\begin{proposition}\label{prop1}
We assume that $W_{\mathrm{rad}}:[0,q]\to[0,\infty)$ is a bounded, nondecreasing, lower semicontinuous function. Let $B_R\subset\R^n$ be the open ball of radius $R>0$ centered at the origin, let
\begin{equation*}
\tilde W_{\mathrm{rad}}(r)=\begin{cases}
W_{\mathrm{rad}}(|r|) &\text{when } |r|\leq q,\\
W_{\mathrm{rad}}(q) &\text{when } |r|\geq q,
\end{cases}
\end{equation*}
and let 
\begin{equation*}
 J_{B_R}(h):=\int_{B_R} \Big[\frac{1}{2}|\nabla h(x)|^2+\tilde W_{\mathrm{rad}}(h(x))\Big]\dd x.
\end{equation*}
Then, there exists a unique minimizer $\overline{\Psi}_R$ (resp. $\underline{\Psi}_R$) of $J_{B_R}$ in the class $A_{R}:=\{h\in W^{1,2}(\Omega;\R): h=q \text{ on } \partial B_R\}$, satisfying the following properties:
\begin{itemize}
\item[(i)] $\overline{\Psi}_R$ (resp. $\underline{\Psi}_R$) is radial (i.e. $\overline{\Psi}_R(x)=\overline{\Psi}_{R,\mathrm{rad}}(|x|)$, $\underline{\Psi}_R(x)=\underline{\Psi}_{R,\mathrm{rad}}(|x|)$, $\forall x\in \overline{ B_R}$), and continuous on $\overline{B_R}$,
\item[(ii)] the function $\overline{\Psi}_{R,\mathrm{rad}}$ (resp. $\underline{\Psi}_{R,\mathrm{rad}}$) is nondecreasing on the interval $[0,R]$,
\item[(iii)] if $\psi_R$ is another minimizer of $J_{B_R}$ in the class $A_{R}$, then we have $\underline{\Psi}_R\leq \psi_R\leq\overline{\Psi}_R$ in $\overline{B_R}$.
\end{itemize}
\end{proposition}

\begin{remark}\label{rem1}
In general, the minimizer of $J_{B_R}$ in the class $A_R$ is not unique. For instance, let us assume that $n=2$, and
\begin{equation*}
\tilde W_{\mathrm{rad}}(r)=\begin{cases}
0 &\text{if } r=0,\\
1 &\text{if } r>0.
\end{cases}
\end{equation*}
Then, a computation (cf. Lemma \ref{lemad}) shows that when $R=R_0:=\sqrt{2e}\, q$, $J_{B_R}$ admits exactly two radial minimizers in the class $A_R$, namely $\overline{\Psi}_R \equiv q$, and
\begin{equation*}
\underline{\Psi}_R(x)=\begin{cases}
0 &\text{if } |x|\leq \sqrt{2}\, q,\\
2q \ln (\frac{|x|}{\sqrt{2}\, q}) &\text{if } \sqrt{2}\, q\leq |x|\leq R_0.
\end{cases}
\end{equation*}
On the other hand, when $R>R_0$ (resp. $R<R_0$), $J_{B_R}$ admits only one radial minimizer in the class $A_R$, namely
\begin{equation*}
\underline{\Psi}_R(x)=\overline{\Psi}_R(x)=\begin{cases}
0 &\text{if }  |x|\leq a_R,\\
q \frac{ \ln (|x|)-\ln a_R}{\ln R- \ln a_R} &\text{if }  |x|\geq a_R,
\end{cases}
\end{equation*}
where $a_R$ is the only solution of $\sqrt{2} \, a\ln(R/a)=q$ in the inteval $(\frac{R}{\sqrt{e}},R)$  (resp. $\underline{\Psi}_R=\overline{\Psi}_R\equiv q$). Thus, Proposition \ref{prop1} implies that this is the only minimizer of $J_{B_R}$ in the class $A_R$.
\end{remark}

We refer to Lemma \ref{pl3a} below, for further properties of the comparison functions $\underline{\Psi}_R$ and $\overline{\Psi}_R$. In particular, we study their dependence on $R$, and we establish that the minimizer of $J_{B_R}$ in the class $A_R$ is unique, for every $R\in (0,\infty)\setminus D$, where $D$ is a countable subset of $(0,\infty)$.

Next, we state the comparison principle:
\begin{theorem}\label{th1}
We assume that hypotheses $\mathbf{H_1}$-$\mathbf{H_3}$ hold, and that the map $u\in W^{1,2}_{\mathrm{loc}}(\Omega;\R^m)$ satisfies \eqref{qbound} and \eqref{enemin}. Let $\overline{\Psi}_R$ be the radial minimizer defined in Proposition \ref{prop1}. Then, 
for every closed ball $\overline{B_R(x_0)}$ contained in $\Omega$, we have
\begin{equation}\label{compp}
|u(x)|\leq \overline{\Psi}_R(x-x_0), \text{ on }\overline{B_R(x_0)}.
\end{equation}
\end{theorem}

\begin{remark}
In the case where $m=1$, $W_0\equiv 0$, and $W(u)=W_{\mathrm{rad}}(|u|)$, the bound provided by Theorem \ref{th1} is optimal, since the function $\overline{\Psi}_R:B_R\to [0,q]$ is a minimizer of \eqref{ene} satisfying \eqref{qbound}.
\end{remark}

\begin{remark}
Theorem \ref{th1} covers the case where the profile of $W$ is not uniform along the rays emanating from the origin. For instance, if we take $W(u)=|u|^{\alpha(u/|u|)}$ in the unit ball $B_1$, with $\alpha:\SF^{m-1}\to(0,\infty)$ a continuous function, then setting $\overline\alpha:=\max_{\SF^{m-1}}\alpha$, and $\underline\alpha:=\min_{\SF^{m-1}}\alpha$, we can apply Theorem \ref{th1} in the ball of radius $q:=e^{-\underline{ \alpha}^{-1}}$, with $W_{\mathrm{rad}}(s)=s^{\overline \alpha}$, $\forall s\in [0,q]$, since the functions $[0,q]\mapsto W(s\xi)-W_{\mathrm{rad}}(s)$ are nondecreasing, for every $\xi\in\SF^{m-1}$.

\end{remark}

\begin{remark}
Let $W_{\mathrm{rad}}:[0,q]\to[0,\infty)$ (resp. $V_{\mathrm{rad}}:[0,q]\to[0,\infty)$) be two bounded, nondecreasing, lower semicontinuous functions, and let $\overline{\Psi}_R$ (resp. $\overline{\Phi}_R$) be the corresponding comparison functions provided by Proposition \ref{prop1}. If moreover we assume that the function $V_{\mathrm{rad}}-W_{\mathrm{rad}}$ is nondecreasing on $[0,q]$, then an application of Theorem \ref{th1} with $u=\overline{\Phi}_R$, and $W(u)=V_{\mathrm{rad}}(|u|)$, shows that $\overline{\Phi}_R\leq \overline{\Psi}_R$ holds on $B_R$. Thus, the optimal comparison function $\overline{\Psi}_R$ provided by Proposition \ref{prop1}, is obtained by choosing the greatest function $W_{\mathrm{rad}}$ satisfying $\mathbf{H_3}$. This also explains why the profile of the comparison function $\overline{\Psi}_R$, corresponding to the potential $W_{\mathrm{rad}}(s)=s^\alpha$ ($\alpha>0$), flattens as $\alpha$ decreases. 
\end{remark}

We also have the following useful version of Theorem \ref{th1} at the boundary of $\Omega$:
\begin{theorem}\label{th1bis}
We assume that hypotheses $\mathbf{H_1}$-$\mathbf{H_3}$ hold. Let $\Omega\subset\R^n$ be a bounded, open set with Lipschitz boundary, and let $u\in W^{1,2}(\Omega;\R^m)$ be a map satisfying \eqref{qbound}, and \eqref{enemin} for every
$v=u+\xi$ such that $\xi \in W^{1,2}_0(\Omega;\R^m)$, and $\|v\|_{L^\infty(\Omega;\R^m)}\leq q$.
Then, if the ball $B_R(x_0)$ intersects $\partial \Omega$, and if $u=0$ on $B_R(x_0)\cap\partial \Omega$, we have
\begin{equation}\label{compp}
|u(x)|\leq \overline{\Psi}_R(x-x_0), \text{ on }\overline{ B_R(x_0)\cap \Omega}.
\end{equation}
\end{theorem}

In Lemma \ref{pl3} below, we determine the conditions implying the existence of dead core regions for the comparison function $\overline{\Psi}_R$. Therefore, by combining Theorem \ref{th1} with Lemma \ref{pl3}, we also give in Theorem \ref{th2} a sufficient condition for the existence of dead core regions\footnote{As a consequence of Theorem \ref{th1bis} and Lemma \ref{pl3}, we also deduce the existence of dead core regions at the boundary of $\Omega$.}, in the case of vector minimizers:

\begin{theorem}\label{th2}
In addition to hypotheses $\mathbf{H_1}$-$\mathbf{H_3}$, we assume that 
\begin{itemize}\label{hypw}
\item[$\mathbf{H_4}$]  $W_{\mathrm{rad}}(s)>0$, $\forall s \in (0,q]$, and $I_q:=\int_0^q\frac{\dd s}{\sqrt{W_{\mathrm{rad}}(s)}}<\infty$.
\end{itemize}
Then, if the map $u\in W^{1,2}_{\mathrm{loc}}(\Omega;\R^m)$ satisfies \eqref{qbound} and \eqref{enemin}, we have $u(x)=0$, provided that $d(x,\partial \Omega)\geq (4n+\sqrt{2})I_q$. 
\end{theorem}

In the scalar case, if hypothesis $\mathbf{H_4}$ does not hold, then the existence of dead cores is partially ruled out by the following variational version of the maximum principle: 
\begin{proposition}\label{prop2}
Let $m=1$, and let $\omega$ be a bounded open set with Lipschitz boundary, such that $\overline \omega\subset\Omega$. We assume that hypotheses $\mathbf{H_1}$-$\mathbf{H_3}$ hold for $W(u)=W_{\mathrm{rad}}(|u|)$ ($W_0\equiv 0$), and moreover that 
\begin{equation}\label{hypwinf}
\int_0^q\frac{\dd s}{\sqrt{W(s)}}=\infty \text{ or $W\equiv 0$ in a neighbourhood of $0$}.
\end{equation}
Then, if the function $u\in W^{1,2}_{\mathrm{loc}}(\Omega;\R)$ satisfies \eqref{qbound}, \eqref{enemin}, and $u > 0$ on $\partial \omega$, we also have $u(x)>0,\, \forall x \in \omega$.
\end{proposition}

The plan of the next sections is as follows. In section \ref{sec:sec3} we give the proofs of Propositions \ref{prop1} and \ref{prop2}, as well as Theorems \ref{th1} and \ref{th1bis}. In section \ref{sec:sec4} we recall that the minimizers of \eqref{ene} are continuous, and we also establish the validity of Pohozaev identity for minimizers of \eqref{ene}. This identity is crucial in the proof of Lemma \ref{pl3}.

\section{Proofs of Propositions \ref{prop1} and \ref{prop2}, and Theorems \ref{th1} and \ref{th1bis}}\label{sec:sec3}

We first establish the existence of a radial minimizer of $J_{B_R}$ in the class $A_{R}$.

\begin{lemma}\label{pl1}
Under the assumptions of Proposition \ref{prop1}:
\begin{itemize}
\item There exists a minimizer $\psi_R$ of $J_{B_R}$ in the class $A_{R}$, which is radial (i.e. $\psi_R(x)=\psi_{R,\mathrm{rad}}(|x|)$, $\forall x\in B_R$), and continuous on $\overline{B_R}$.
\item For such a radial minimizer, the function $\psi_{R,\mathrm{rad}}$ is nondecreasing on the interval $[0,R]$.
\end{itemize}
\end{lemma}

\begin{proof}  
Let $\tilde \psi$ be a minimizer of $J_{B_R}$ in the class $A_R:=\{h\in W^{1,2}(\Omega;\R): \ h=q \text{ on } \partial B_R\}$. We first notice that $0\leq \tilde\psi\leq q$, since otherwise the competitor $\min (\tilde \psi^+,q)\in A_R$ has less energy. We also know that $\tilde\psi$ is continuous in $B_R$ (cf. Lemma \ref{lem2}). Starting from $\tilde \psi$, we can construct 
\begin{equation*}
\tilde \psi_0^{(1)}(x)=\tilde \psi(|x_1|,x_2,\ldots,x_n),
\end{equation*}
which is another minimizer of $J_{B_R}$ in $A_R$, invariant by the reflection $(x_1,x_2,\ldots,x_n)\mapsto(-x_1,x_2,\ldots,x_n)$. 
Indeed, we have $ J_{B_R\cap \{x_1>0\}}(\tilde\psi)= J_{B_R\cap \{x_1<0\}}(\tilde\psi) $, since otherwise either the competitor $x\mapsto \tilde \psi(-|x_1|,x_2,\ldots,x_n)$ or the competitor $\tilde \psi_0^{(1)}$ has less energy than $\tilde\psi$. 
Similarly, we can construct a minimizer 
\begin{equation*}
\tilde \psi_1^{(1)}(x)=\tilde \psi(|x_1|,|x_2|,\ldots,x_n),
\end{equation*}
which coincides with $\tilde\psi$ on $\{ x\in B_R: x_1>0, x_2>0\}$, and is invariant by the reflections $(x_1,x_2,\ldots,x_n)\mapsto(-x_1,x_2,\ldots,x_n)$ and $(x_1,x_2,\ldots,x_n)\mapsto(x_1,-x_2,\ldots,x_n)$. By repeating this process, we obtain for every $k\geq 2$, a minimizer $\tilde\psi_k^{(1)}$, which coincides with $\tilde\psi$ on $\{ x\in B_R:0< x_2<\tan(\pi/2^k)x_1\}$, and is invariant by the dihedral group $D_{2^k}$ generated by the reflections with respect to the hyperlanes $x_2=0$, and $x_2=\tan(\pi/2^k)x_1$. It is clear that $\|\tilde \psi_k^{(1)}\|_{W^{1,2}(B_R;\R)}$ is uniformly bounded, thus (up to subsequence) we have 
\begin{equation*}
\tilde\psi_k^{(1)} \rightharpoonup \tilde\psi_\infty^{(1)} \text{ in } W^{1,2}(B_R;\R), \text{ and } \tilde\psi_k^{(1)} \to\tilde\psi_\infty^{(1)} \text{ a.e. in } B_R.
\end{equation*}
By the weakly lower continuity of the $L^2$ norm, it follows that
\begin{subequations}
\begin{equation}
\int_{B_R} |\nabla\tilde  \psi_\infty^{(1)}|^2\leq \liminf_{k\to\infty}\int_{B_R} |\nabla \tilde\psi_k^{(1)}|^2,
\end{equation}
while by Fatou's lemma and the lower semicontinuity of $\tilde W_{\mathrm{rad}}$, we get
\begin{equation}
\int_{B_R} \tilde W_{\mathrm{rad}}(\tilde\psi_\infty^{(1)})\leq \int_{B_R}\liminf_{k\to\infty}\tilde W_{\mathrm{rad}} (\tilde \psi_k^{(1)}) \leq \liminf_{k\to\infty}\int_{B_R}\tilde W_{\mathrm{rad}} (\tilde \psi_k^{(1)}).
\end{equation}
\end{subequations}
As a consequence, $\tilde\psi_\infty^{(1)}$ is another minimizer of $J_{B_R}$ in $A_R$. 
By construction, given $x\in B_R$, such that $l_{12}:=\sqrt{x_1^2+x_2^2}$, we have 
\begin{multline*}
|\tilde\psi_k^{(1)}(x_1,x_2,x_3,\ldots,x_n)-\tilde\psi(l_{12},0,x_3,\ldots,x_n)|\leq \\
\sup \Big\{|\tilde \psi(z_1,z_2,x_3,\ldots,x_n)-\tilde\psi(l_{12},0,x_3,\ldots,x_n)|: \sqrt{(z_1^2-l_{12})^2+z_2^2}\leq \frac{\pi l_{12}}{2^k}\Big\}.
\end{multline*}
Therefore, letting $k\to\infty$, it follows that $\tilde\psi_\infty^{(1)}(x_1,x_2,x_3,\ldots,x_n)=\tilde\psi(\sqrt{x_1^2+x_2^2},0,x_3,\ldots,x_n)$. 

Next, we proceed by induction, and starting from $\tilde\psi_\infty^{(1)}$, we consider for every $k\geq 2$, a minimizer $\tilde\psi_k^{(2)}$, which coincides with $\tilde\psi_\infty^{(1)}$ on $\{ x\in B_R:0< x_3<\tan(\pi/2^k)x_1\}$, and is invariant by the dihedral group $D_{2^k}$ generated by the reflections with respect to the hyperlanes $x_3=0$, and $x_3=\tan(\pi/2^k)x_1$. As previously $\tilde\psi_\infty^{(2)}:=\lim_{k\to\infty}\tilde\psi_k^{(2)}$ is a minimizer of $J_{B_R}$ in $A_R$, such that
$\tilde\psi_\infty^{(2)}(x_1,x_2,x_3,x_4,\ldots,x_n)=\tilde\psi_\infty^{(1)}(\sqrt{x_1^2+x_3^2},x_2,0,x_4,\ldots,x_n)=\tilde\psi(\sqrt{x_1^2+x_2^2+x_3^2},0,0,x_4,\ldots,x_n) $. The process terminates after a finite number of exactly $n-1$ steps, when we get a minimizer $\psi_R:=\tilde \psi_\infty^{(n-1)}$ of $J_{B_R}$ in $A_R$, such that $\psi_R(x_1,x_2,\ldots,x_n)=\tilde\psi(\sqrt{x_1^2+\ldots+x_n^2},0,\ldots,0)$.

Given a radial radial minimizer $\psi_R$ of $J_{B_R}$ in the class $A_{R}$, we can easily see by contradiction that the function $\psi_{R,\mathrm{rad}}$ is nondecreasing on the interval $[0,R]$. Indeed, assume that $\psi_{R,\mathrm{rad}}(r)>\psi_{R,\mathrm{rad}}(s)$ holds for some $0\leq r<s\leq R$. Then, the competitor
\begin{equation}
\zeta(x):=\begin{cases}
\psi_R(x) &\text{if } s\leq |x|\leq R,\\
\min(\psi_R(x),\psi_{R,\mathrm{rad}}(s)) &\text{if }|x|\leq s,
\end{cases}
\end{equation}
has less energy than $\psi_R$, which is impossible.
Finally, in view of the monotonicity of $\psi_{R,\mathrm{rad}}$, the continuity of $\psi_R$ up to $\overline{B_R}$ is clear.
\end{proof}

In the next Lemma, we consider a perturbation of the functional $J_{B_R}$ for $\lambda\in (0,1)$ (cf. \eqref{enel}). We shall use the corresponding comparison functions $\psi_R^\lambda$ provided by Lemma \ref{pl1}, to obtain an upper bound for the modulus of the local minimizer $u$ considered in Theorem \ref{th1}.

\begin{lemma}\label{pl2}
We assume that hypotheses $\mathbf{H_1}$-$\mathbf{H_3}$ hold, and that the map $u\in W^{1,2}_{\mathrm{loc}}(\Omega;\R^m)$ satisfies \eqref{qbound} and \eqref{enemin}.
Given $\lambda \in (0,1)$, let
\begin{equation}\label{enel}
J_{B_R}^\lambda (h):=\int_{B_R} \Big[\frac{1}{2}|\nabla h(x)|^2+\lambda \tilde W_{\mathrm{rad}}(h(x))\Big]\dd x,
\end{equation}
and consider a radial minimizer $\psi_R^\lambda $ of $ J_{B_R}^\lambda$ in the class $A_R$, provided by Lemma \ref{pl1}. 
Then, for every closed ball $\overline{B_R(x_0)}$ contained in $\Omega$, we have
\begin{equation}\label{comppnew}
|u(x)|\leq \psi_R^\lambda(x-x_0), \text{ on }\overline{B_R(x_0)}.
\end{equation}
\end{lemma}

\begin{proof}     
Without loss of generality, we assume that $x_0=0$. We recall that $u$ is continuous on $\Omega$ (cf. Lemma \ref{lem2}), and consider on the open set $\Omega_0:=\{x\in\Omega: u(x)\neq 0\}$ the polar form:
\begin{equation}
u(x)=\rho(x)\n(x), \text{ with } \rho(x):=|u(x)|, \ \n(x):=\frac{u(x)}{|u(x)|}.
\end{equation}
An easy computation shows that
\begin{equation}
|\nabla u(x)|^2=|\nabla\rho(x)|^2+|\rho(x)|^2|\nabla \n(x)|^2 \text{ on } \Omega_0.
\end{equation}

Next, we define on $\overline{B_R}$ the comparison map:
\begin{equation}\label{compu}
\tilde u(x)=\begin{cases}
u(x) &\text{when } \rho(x)\leq\psi_R^\lambda(x)\\
\psi_R^\lambda(x)\n &\text{when } \rho(x)>\psi_R^\lambda (x),\end{cases}
\end{equation}
where $\psi_R^\lambda$ is a radial minimizer of $ J_{B_R}^\lambda$ in the class $A_R$, provided by Lemma \ref{pl1}. 
It is obvious that $u=\tilde u$ on $\partial B_R$. One can also check that $|\tilde u|\leq|u|$ holds on $\overline{B_R}$, and $\tilde u$ is continuous on $\overline{B_R}$. Our claim is that $\tilde u \in W^{1,2}(B_R;\R^m)$. Let $U:=\{x\in B_R: \psi_R(x)>0\}$. We notice that either 
$U=B_R$, or $U=\{x: R'<|x|<R\}$, for some $R'\in (0,R)$. Now, given $x\in \Omega_0\cap U$, it is clear that $\tilde u(x)=\frac{\min(\psi_R(x),\rho(x))}{\rho(x)}u(x)$ holds in an open neighbourhood $V_x$ of $x$, where $\rho(x)\geq\epsilon>0$. As a consequence, $\tilde u \in W^{1,2}(V_x;\R^m)$, as a product of maps belonging to $W^{1,2}(V_x;\R^m)\cap L^\infty(V_x;\R^m)$. Otherwise, if $u$ vanishes for some $x\in U$, we have $\tilde u=u$ in a neighbourhood of $x$. This proves that $\tilde u \in W^{1,2}_{\mathrm{loc}}(U;\R^m)$. Moreover, setting $\tilde\rho(x)=|\tilde u(x)|=\min(\psi_R(x),\rho(x))$, we compute
\begin{align*}
\int_U|\nabla \tilde u|^2&=\int_{\Omega_0\cap U}(|\nabla\tilde\rho|^2+\tilde\rho^2|\nabla n|^2) \\
&\leq \int_{\Omega_0\cap U}( |\nabla\rho|^2+\rho^2|\nabla n|^2)+\int_{B_R}|\nabla \psi_R|^2=\int_U|\nabla u|^2+\int_{B_R}|\nabla \psi_R|^2<\infty ,
\end{align*}
thus $\tilde u \in W^{1,2}(U;\R^m)$. Finally, in the case where $U\neq B_R$ i.e. $U=\{x: R'<|x|<R\}$, we have $\tilde u\equiv 0$ in $\overline{B_{R'}}$. This proves our claim that $\tilde u \in W^{1,2}(B_R;\R^m)$.

At this stage, we utilize the minimality of $u$ to deduce that
\begin{align*}
E_{B_R}(u)&=E_{B_R\cap\{\rho>\psi_R^\lambda\}}(u)+E_{B_R\cap\{\rho\leq \psi_R^\lambda\}}(u)\\
&\leq E_{B_R\cap\{\rho>\psi_R^\lambda\}}(\tilde u)+E_{B_R\cap\{\rho\leq \psi_R^\lambda\}}(u)=E_{B_R}(\tilde u),
\end{align*}
or equivalently
\begin{multline}\label{ga1}
E_{B_R\cap\{\rho>\psi_R^\lambda\}}(u)=\int_{B_R\cap\{\rho>\psi_R^\lambda\}}\Big[\frac{|\nabla\rho|^2}{2}+\frac{|\rho|^2|\nabla \n|^2}{2}+W_{\mathrm{rad}}(\rho)+W_0(\rho\n)\Big]\\
\leq \int_{B_R\cap\{\rho>\psi_R^\lambda\}}\Big[\frac{|\nabla\psi_R^\lambda|^2}{2}+\frac{|\psi_R^\lambda|^2|\nabla \n|^2}{2}+W_{\mathrm{rad}}(\psi_R^\lambda)+W_0(\psi_R^\lambda\n)\Big]=E_{B_R\cap\{\rho>\psi_R^\lambda\}}(\tilde u).
\end{multline}
Similarly, by the minimality of $\psi_R^\lambda$, and since $(\rho-\psi_R^\lambda)^+\in W^{1,2}_0(B_R)$, it follows that
\begin{align*}
J_{B_R}^\lambda(\psi_R^\lambda)&=J_{B_R\cap\{\rho>\psi_R^\lambda\}}^\lambda(\psi_R^\lambda)+J_{B_R\cap\{\rho\leq \psi_R^\lambda\}}^\lambda(\psi_R^\lambda)\\
&\leq J_{B_R\cap\{\rho>\psi_R^\lambda\}}^\lambda(\rho)+J_{B_R\cap\{\rho\leq \psi_R^\lambda\}}^\lambda(\psi_R^\lambda)=J_{B_R}^\lambda(\psi_R^\lambda+(\rho-\psi_R^\lambda)^+),
\end{align*}
or equivalently
\begin{align}\label{ga2}
J_{B_R\cap\{\rho>\psi_R^\lambda\}}^\lambda(\psi_R^\lambda)&=\int_{B_R\cap\{\rho>\psi_R^\lambda\}}\Big[\frac{|\nabla\psi_R^\lambda|^2}{2}+\lambda W_{\mathrm{rad}}(\psi_R^\lambda)\Big]\\
&\leq \int_{B_R\cap\{\rho>\psi_R^\lambda\}}\Big[\frac{|\nabla\rho|^2}{2}+\lambda W_{\mathrm{rad}}(\rho)\Big]=J^\lambda_{B_R\cap\{\rho>\psi_R^\lambda\}}(\rho).\nonumber
\end{align}
Gathering the previous results from \eqref{ga1} and \eqref{ga2}, we conclude that
\begin{subequations}
\begin{equation}\label{ii1}
I_1+I_2+I_3+I_4\leq 0
\end{equation}
with
\begin{equation}\label{ii2}
I_1:=\int_{B_R\cap\{\rho>\psi_R^\lambda\}}\Big[\frac{|\nabla\rho|^2}{2}+\lambda W_{\mathrm{rad}}(\rho)-\frac{|\nabla\psi_R^\lambda|^2}{2}-\lambda W_{\mathrm{rad}}(\psi_R^\lambda)\Big]\geq 0 \text{ (cf. \eqref{ga2})},
\end{equation}
\begin{equation}\label{ii3b}
I_2:=\int_{B_R\cap\{\rho>\psi_R^\lambda\}}\frac{(|\rho|^2-|\psi_R^\lambda|^2)}{2}|\nabla \n|^2 \geq 0,
\end{equation}
\begin{equation}\label{ii4}
I_3:=\int_{B_R\cap\{\rho>\psi_R^\lambda\}}(W_0(\rho\n)-W_0(\psi_R^\lambda\n))\geq  0 \text{ (cf. $\mathbf{H_3}$)},
\end{equation}
\begin{equation}\label{ii5b}
I_4:=(1-\lambda)\int_{B_R\cap\{\rho>\psi_R^\lambda\}}(W_{\mathrm{rad}}(\rho)-W_{\mathrm{rad}}(\psi_R^\lambda))\geq 0 \text{ (cf. $\mathbf{H_3}$)}.
\end{equation}
\end{subequations}
Consequently, we have
\begin{subequations}
\begin{equation}\label{ii3}
I_2:=\int_{B_R\cap\{\rho>\psi_R^\lambda\}}\frac{(|\rho|^2-|\psi_R^\lambda|^2)}{2}|\nabla \n|^2 = 0,
\end{equation}
\begin{equation}\label{ii5}
I_4:=(1-\lambda)\int_{B_R\cap\{\rho>\psi_R^\lambda\}}(W_{\mathrm{rad}}(\rho)-W_{\mathrm{rad}}(\psi_R^\lambda))= 0.
\end{equation}
\end{subequations}

Now, let us assume by contradiction that the open set $V:=B_R\cap\{\rho>\psi_R^\lambda\}$ is nonempty, and let $\tilde V$ be a nonempty connected component of $V$. It follows from \eqref{ii3} that $\nabla \n\equiv 0$ holds in $\tilde V$, thus we have $\n \equiv \n_0$ in $\tilde V$, for a unit vector $\n_0\in \R^m$, as well as $u=\rho \n_0 $ in $\tilde V$. Our next claim is that 
\begin{equation}\label{claim2}
W_{\mathrm{rad}}(\psi_R^\lambda)=W_{\mathrm{rad}}(\rho)\equiv \mathrm{Const.} \text{ in } \tilde V.
\end{equation}
Indeed, let us first assume by contradiction that $W_{\mathrm{rad}}(\psi_R^\lambda(x_0))+2\epsilon<W_{\mathrm{rad}}(\rho(x_0))$ holds for some $x_0\in \tilde V$, and $\epsilon>0$. Then, by the lower semicontinuity of $W_{\mathrm{rad}}(\rho)$, we have $W_{\mathrm{rad}}(\psi_R^\lambda(x_0))+\epsilon<W_{\mathrm{rad}}(\rho)$ in an open neighbourhood $\tilde V_{x_0}\subset \tilde V$ of $x_0$. On the other hand, since $W_{\mathrm{rad}}$ is nondecreasing on $[0,q]$, while $|x|\mapsto\psi_R^\lambda(|x|)$ is nondecreasing on $[0,R]$, it is clear that  $W_{\mathrm{rad}}(\psi_R^\lambda)\leq W_{\mathrm{rad}}(\psi_R^\lambda(x_0))$ holds on the set $S:= \{x\in \tilde V_{x_0}: |x|\leq |x_0|\}$ (which has positive Lebesgue measure). As a consequence, we have $W_{\mathrm{rad}}(\rho)-W_{\mathrm{rad}}(\psi_R^\lambda))\geq \epsilon>0$ on $S$, in contradiction with \eqref{ii5}. This proves that $W_{\mathrm{rad}}(\psi_R^\lambda)\equiv W_{\mathrm{rad}}(\rho)$ in $\tilde V$. Next, 
we assume by contradiction that $W_{\mathrm{rad}}(\psi_R^\lambda(x_1))<W_{\mathrm{rad}}(\psi_R^\lambda(x_2))$ holds for some $x_1,x_2\in \tilde V$.  Let $q_1:=\psi_R^\lambda(x_1)$, $q_2:=\psi_R^\lambda(x_2)$ (with $q_1<q_2$, since $W_{\mathrm{rad}}$ is nondecreasing), and let $s:=\max\{ r\in [0,q]: W_{\mathrm{rad}}(r)=q_1\}$. We notice that $s\in [q_1,q_2]$, thus in view of the continuity of $\psi_R^\lambda$, there exists $x_3\in \tilde V$ such that $\psi_R^\lambda(x_3)=s$. By definition of $s$, we have $W_{\mathrm{rad}}(\psi_R^\lambda(x_3))<W_{\mathrm{rad}}(\rho(x_3))$, which is a contradiction. This establishes \eqref{claim2}.

To prove the bound
\begin{equation}\label{comp2}
|u(x)|\leq \psi_R^\lambda(x), \text{ on }\overline{B_R},
\end{equation}
it remains to show that
\begin{equation}\label{claim3}
\Delta \psi_R^\lambda \leq 0, \text{ and } \Delta \rho \geq 0 \text{ in } \tilde V.
\end{equation}
Indeed, since the boundary condition $\rho-\psi_R^\lambda\leq 0$ is satisfied on $\partial \tilde V$, the maximum principle would give that $\rho\leq \psi_R^\lambda$ holds in $\tilde V$, in contradiction with our assumption that $\tilde V$ is nonempty. To check \eqref{claim3}, we utilize the minimality of $u$ and $\psi_R^\lambda$, as well as \eqref{claim2}. Given $x_0\in \tilde V$, let $s:=\psi_R^\lambda(x_0)$, $t:=\rho (x_0)$, and $2\kappa:=t-s>0$. In view of \eqref{claim3} is is clear that $W_{\mathrm{rad}}$ is constant on $[s,t]$. Let also $\tilde V_{x_0}\subset \tilde V$ be an open neighbourhood of $x_0$, such that  $\psi_R^\lambda\leq s+\kappa$ and $\rho\geq t-\kappa$ hold in $\tilde V_{x_0}$. Now, given $\phi\in C^1_0(\R^n;\R)$, such that $\supp\phi\subset\tilde V_{x_0}$, and $0\leq\phi\leq \kappa$, we have for every $\epsilon \in (0,1)$:
\begin{subequations}\label{harm2}
\begin{equation}
\frac{J_{B_R}^\lambda (\psi_R^\lambda+\epsilon\phi)-J_{B_R}^\lambda (\psi_R^\lambda)}{\epsilon}=\int_{B_R}\frac{|\nabla\psi_R^\lambda+\epsilon\nabla\phi|^2-|\nabla\psi_R^\lambda|^2}{2\epsilon}\geq 0,
\end{equation}
\begin{equation}
\frac{E_{B_R} (u-\epsilon\phi\n_0)-E_{B_R}^\lambda (u)}{\epsilon}=\int_{B_R}\frac{|\nabla\rho-\epsilon\nabla\phi|^2-|\nabla\rho|^2}{2\epsilon}+\int_{B_R}(W_0((\rho-\epsilon\phi)\n_0)-W_0(\rho\n_0))\geq 0.
\end{equation}
\end{subequations}
Finally, since $W_0((\rho-\epsilon\phi)\n_0)\leq W_0(\rho\n_0)$, we let $\epsilon \to 0$ in \eqref{harm2}, and deduce that
\begin{subequations}\label{harm3}
\begin{equation}
\int_{B_R}\nabla\psi_R^\lambda\cdot\nabla\phi\geq 0,\text{ i.e. $\psi_R^\lambda$ is superharmonic in $\tilde V_{x_0}$,}
\end{equation}
\begin{equation}
\int_{B_R}\nabla\rho\cdot\nabla\phi\leq 0,\text{ i.e. $\rho$ is subharmonic in $\tilde V_{x_0}$.}
\end{equation}
\end{subequations}
This establishes \eqref{claim3}, and completes the proof of \eqref{comp2}.
\end{proof}

Now, we are able to complete the proofs of Proposition \ref{prop1}, as well as Theorems \ref{th1} and \ref{th1bis}.

\begin{proof}[Proof of Proposition \ref{prop1}]
For every $\lambda >0$, let $\psi_R^\lambda $ be a radial minimizer of $ J_{B_R}^\lambda$ in the class $A_R$, provided by Lemma \ref{pl1}.
We first notice that $\psi_R^\lambda$ is uniformly bounded in $W^{1,2}(B_R)$, provided that $\lambda$ remains bounded. Thus, as $\lambda\to 1$ and $\lambda <1$, we have (up to subsequence):
\begin{equation}\label{lim11}
\psi_R^{\lambda} \rightharpoonup \overline{\zeta} \text{ in } W^{1,2}(B_R;\R), \text{ and } \psi_R^{\lambda} \to \overline{\zeta} \text{ a.e. in } B_R.
\end{equation}
In addition, by the weakly lower continuity of the $L^2$ norm, it follows that
\begin{subequations}
\begin{equation}
\int_{B_R} |\nabla \overline{\zeta}|^2\leq \liminf_{\lambda\to 1^-}\int_{B_R} |\nabla \psi_R^\lambda|^2,
\end{equation}
while by Fatou's lemma and the lower semicontinuity of $W_{\mathrm{rad}}$, we get
\begin{equation}
\int_{B_R} W_{\mathrm{rad}}(\overline{\zeta})\leq \int_{B_R}\liminf_{\lambda\to 1^-} W_{\mathrm{rad}} (\psi_R^{\lambda}) \leq \liminf_{\lambda\to 1^-}\int_{B_R}\lambda W_{\mathrm{rad}} ( \psi_R^\lambda).
\end{equation}
\end{subequations}
Finally, in view of the minimality of $\psi_R^\lambda$, we deduce that
\begin{subequations}
\begin{equation}
J_{B_R}^\lambda (\psi_R^1)\geq J_{B_R}^\lambda (\psi_R^\lambda),
\end{equation}
\begin{equation}
J_{B_R}(\psi_R^1)= \liminf_{\lambda\to 1^-}J_{B_R}^\lambda (\psi_R^1)\geq  \liminf_{\lambda\to 1^-} J_{B_R}^\lambda (\psi_R^\lambda) \geq J_{B_R}(\overline{\zeta}).
\end{equation}
\end{subequations}
That is, $\overline{\zeta}$ is a minimizer of $J_{B_R}$ in the class $A_R$. Moreover, by construction $\overline{\zeta}$ is radial, and nondecreasing as a function of $|x|$. It remains to establish that $\overline{\zeta}$ also satisfies property (iii) of Proposition \ref{prop1}. Indeed, if $\tilde \psi$ is another minimizer of $ J_{B_R}$ in $A_R$, we have in view of Lemma \ref{pl2} applied to $\tilde \psi$ instead of $u$:
\begin{equation}
\tilde \psi(x)\leq \psi_R^\lambda(x), \forall x\in \overline{B_R}, \forall \lambda\in (0,1)\Rightarrow \tilde \psi(x)\leq \overline{\zeta}(x),  \forall x\in \overline{B_R}.
\end{equation}
Therefore, $\overline{\zeta}$ is the minimizer $\overline{\Psi}_R$ described in Proposition \ref{prop1}, which is uniquely determined by property (iii).

Similarly, by taking the limit of the minimizers $\psi_R^\lambda$, as $\lambda \to 1$ and $\lambda>1$, we obtain a radial minimizer $\underline{\zeta}$ of $J_{B_R}$ in the class $A_R$. By construction, we have (up to subsequence):
\begin{equation}\label{lim22}
\lim_{\lambda\to 1^+}\psi_R^{\lambda} = \underline{\zeta} \text{ a.e. in } B_R.
\end{equation}
It remains to establish that $\underline{\zeta}$ also satisfies property (iii) of Proposition \ref{prop1}. To see this, let $\tilde \psi$ be another minimizer of $ J_{B_R}$ in $A_R$. As in the proof of Lemma \ref{pl1}, we can construct for every unit vector $\nu \in \R^n$, a radial minimizer $\tilde \psi_\nu$ of $ J_{B_R}$ in $A_R$, such that $\tilde \psi(s \nu)=\tilde \psi_\nu(s\nu)$, $\forall s \in[0,R]$. Next, we apply Lemma \ref{pl2} with the radial comparison function $\tilde \psi_\nu$, and the potential 
$$W(u)=\lambda W_{\mathrm{rad}}(|u|),  \ W_0(u)=(\lambda -1) W_{\mathrm{rad}}(|u|), \ \lambda>1,$$
to $u=\psi_R^\lambda$, ($\lambda>1$), and deduce that:
\begin{equation}
\psi_R^\lambda(x)\leq \tilde \psi_\nu(x), \forall x\in \overline{B_R}, \forall \lambda>1, \forall \nu \in\SF^{n-1}\Rightarrow \underline{\zeta}(x)\leq \tilde \psi(x),  \forall x\in \overline{B_R}.
\end{equation}
Therefore, $\underline{\zeta}$ is the minimizer $\underline{\Psi}_R$ described in Proposition \ref{prop1}, which is uniquely determined by property (iii).
\end{proof}

\begin{proof}[Proof of Theorem \ref{th1}]
The desired bound \eqref{compp} follows by letting $\lambda\to 1$ (with $\lambda<1$) in \eqref{comppnew}, and using \eqref{lim11}.
\end{proof}

\begin{proof}[Proof of Theorem \ref{th1bis}]
We consider on $B_R\cap\Omega$ the comparison map:
\begin{equation}\label{compu}
\tilde u(x)=\begin{cases}
u(x) &\text{when } \rho(x)\leq\psi_R^\lambda(x)\\
\psi_R^\lambda(x)\n &\text{when } \rho(x)>\psi_R^\lambda (x),\end{cases}
\end{equation}
and reproduce the arguments in the proof of Theorem \ref{th1}.
\end{proof}

From Lemma \ref{pl2} and Proposition \ref{prop1}, we also deduce the following useful result:
\begin{lemma}\label{pl3a}
For every $R>0$, and $\lambda>0$, we consider the functional $$J_{B_R}^\lambda (h):=\int_{B_R} \Big[\frac{1}{2}|\nabla h(x)|^2+\lambda \tilde W_{\mathrm{rad}}(h(x))\Big]\dd x,$$ and the corresponding comparison functions $\underline{\Psi}_R^\lambda$ and $\overline{\Psi}_R^\lambda$ provided by Proposition \ref{prop1}. Then, we have
\begin{itemize}
\item[(a)]  $\underline{\Psi}_R^\lambda(x)=\underline{\Psi}_{\frac{R}{\kappa}}^{\kappa^2\lambda}(\frac{x}{\kappa})$, and $\overline{\Psi}_R^\lambda(x)=\overline{\Psi}_{\frac{R}{\kappa}}^{\kappa^2\lambda}(\frac{x}{\kappa})$, $\forall x \in B_R$, $\forall \kappa,\lambda>0$.
\item[(b)]  $\underline{\Psi}_R^\mu\leq\overline{\Psi}_R^\mu\leq \underline{\Psi}_R^\lambda\leq \overline{\Psi}_R^\lambda$,  provided that $0<\lambda< \mu$.
\item[(c)] $\overline{\Psi}_{\frac{R}{\kappa}}^1(\frac{x}{\kappa})\leq \underline{\Psi}_R^1(x)$, $\forall x\in B_R$, $\forall \kappa \in (0,1)$.
\item[(d)] There exists a countable set $D\subset (0,\infty)$, such that for every $R\in (0,\infty)\setminus D$, we have $\underline{\Psi}_R=\overline{\Psi}_R$, and thus the minimizer of $J_{B_R}$ in the class $A_R$ is unique.
\end{itemize}
 \end{lemma}
\begin{proof}
(a) follows from a simple rescaling argument. On the other hand, when $\lambda\in (0,1)$, an application of Lemma \ref{pl2} in $B_R$, with $\underline{\Psi}_R^\lambda$ (instead of $\psi_R^\lambda$), and $u=\overline{\Psi}_R^1$, gives the inequality
 $\overline{\Psi}_R^1\leq \underline{\Psi}_R^\lambda$, from which we derive (b) in view of the rescaling in (a). The proof of (c) is obvious from (a) and (b). Next, let $Q$ be a countable dense subset of the unit ball $B_1$. If $\underline{\Psi}_1^\lambda\neq\overline{\Psi}_1^\lambda$, for some $\lambda_0>0$, then there exists $x_0\in Q$, such that the function $(0,\infty)\ni \lambda\mapsto \overline{\Psi}_1^\lambda(x_0)$ is discontinuous at $\lambda_0$. Let $\tilde D$ be the set of $\lambda_0>0$ such that the function $ \lambda\mapsto \overline{\Psi}_1^\lambda(x_0)$ is discontinuous at $\lambda_0$, for some $x_0\in Q$. We notice that $\tilde D$ is countable, since the functions $(0,\infty)\ni \lambda\mapsto \overline{\Psi}_1^\lambda(x_0)$ are nonincreasing. Thus, for $\lambda\in (0,\infty)\setminus D$, we have 
$\underline{\Psi}_1^\lambda\equiv\overline{\Psi}_1^\lambda \Leftrightarrow \underline{\Psi}_{\sqrt{\lambda}}^1\equiv\overline{\Psi}_{\sqrt{\lambda}}^1$, and this proves (d).

\end{proof}

At this stage, we determine in Lemma \ref{pl3} below, the conditions implying the existence of dead core regions for the comparison function $\overline{\Psi}_R$. The proofs of Theorem \ref{th2} and Proposition \ref{prop2} follow immediately from Lemma \ref{pl3}.
\begin{lemma}\label{pl3}
In addition to the assumptions of Proposition \ref{prop1}, we suppose that $W_{\mathrm{rad}}(s)>0$, $\forall s \in (0,q]$. Then, 
\begin{itemize}
\item if $\int_0^q\frac{\dd s}{\sqrt{W_{\mathrm{rad}}(s)}}=\infty$, we have $\underline{\Psi}_R>0$, $\forall x \in\overline{B_R}$, 
\item if $I_q:=\int_0^q\frac{\dd s}{\sqrt{2W_{\mathrm{rad}}(s)}}<\infty$, the function $\overline{\Psi}_R$ vanishes in the ball $\overline{B_{R-\sqrt{2}I_q}}$, provided that $R\geq(4n+\sqrt{2})I_q$.
\end{itemize}
\end{lemma}
\begin{proof}
In the case where $I_q:=\int_0^q\frac{\dd s}{\sqrt{2 W_{\mathrm{rad}}(s)}}<\infty$, we define the function
$$\gamma: [0,q]\to [0,I_q], \ \gamma(s)=\int_0^s\frac{1}{\sqrt{2 W_{\mathrm{rad}}}}.$$
Since $\gamma$ is strictly increasing, we denote its inverse function by $\beta:=\gamma^{-1}$, $\beta:  [0,I_q]\to[0,q]$, and it is clear that $s_2-s_1\leq \sqrt{2 W_{\mathrm{rad}}(q)}(\gamma(s_2)-\gamma(s_1))$ holds for $0<s_1\leq s_2\leq q$. Thus, $\beta$ belongs to $W^{1,\infty}(0,I_q)$. In addition, we have
$$\gamma'(s)=\frac{1}{\sqrt{2W_{\mathrm{rad}}(s)}} \text{ for a.e. } s \in (0,q), \text{ and } \beta'(t)=\sqrt{2 W_{\mathrm{rad}}(\beta(t))} \text{ for a.e. } t \in (0,I_q).$$
Next, we consider the restriction of the minimizer $\overline{\Psi}_R$ to the ball $B_r\subset \R^ n$ (with $I_q<r <R$), and setting 
\begin{equation}
\phi_r(x):=\begin{cases}
\beta (|x|-r +\gamma(\overline{\Psi}_R(r))) &\text{ if } r -\gamma(\overline{\Psi}_R(r))\leq |x|\leq r,\\
0 &\text{ if } |x|\leq r -\gamma(\overline{\Psi}_R(r)),
\end{cases}
\end{equation}
we obtain a function $\phi_r \in W^{1,2}(B_r)$ such that $\phi_r=\overline{\Psi}_R$ on $\partial B_r$. A computation shows that
\begin{align*}
J_{B_r}(\phi_r)&=| \SF^{n-1}|\int_0^{\gamma(\overline{\Psi}_R(r)) } (t+r-\gamma(\overline{\Psi}_R(r)))^{n-1}  2W_{\mathrm{rad}}(\beta(t)) \dd t\\
&\leq 2 |\SF^{n-1}| W_{\mathrm{rad}}(\overline{\Psi}_R(r)) I_q r^{n-1}, 
\end{align*}
where $ |\SF^{n-1}|$ denotes the measure of the unit sphere $\SF^{n-1}\subset \R^n$. 
On the other hand, Pohozaev identity \eqref{poz} applied to $\overline{\Psi}_R$ in the ball $B_r$ implies that
\begin{align*}
|\SF^{n-1}| r^n \big(W_{\mathrm{rad}}(\overline{\Psi}_{R,\mathrm{rad}}(r))- \frac{1}{2} |\overline{\Psi}'_{R,\mathrm{rad}}(r)|^2\big)\leq nJ_{B_r}(\overline{\Psi}_R)\leq n J_{B_r}(\phi_r), \text{ for a.e. } r \in (I_q,R),
\end{align*}
where in the last inequality we have used the minimality of $\overline{\Psi}_R$. Therefore, we deduce that
\begin{align*}
\big(W_{\mathrm{rad}}(\overline{\Psi}_{R,\mathrm{rad}}(r))- \frac{1}{2} |\overline{\Psi}'_{R,\mathrm{rad}}(r)|^2\big)\leq  2n I_q r^{-1} W_{\mathrm{rad}}(\overline{\Psi}_{R,\mathrm{rad}}(r)), \text{ for a.e. } r \in (I_q,R).
\end{align*}
In particular, for a.e. $r \in (4n I_q, R)$, we have $W_{\mathrm{rad}}(\overline{\Psi}_{R,\mathrm{rad}}(r))\leq |\overline{\Psi}'_{R,\mathrm{rad}}(r)|^2$. Now, let $(l,R)$ be the intersection of the intervals $(4n I_q, R)$ and $\{ r\in (0,R): \overline{\Psi}_{R,\mathrm{rad}}(r)>0\}$. Since $\overline{\Psi}_{R,\mathrm{rad}}$ is strictly increasing on the interval $(l,R)$, we denote its inverse function by $\chi_R:(\delta,q)\to (l,R)$. Proceeding as previously, we can see that given $0<\epsilon \ll 1$, the function  $\chi_R$ is Lipschitz on $(\delta+\epsilon,q)$, and moreover the inequality $\chi'_R(s)\leq\frac{1}{\sqrt{W_{\mathrm{rad}}(s)}}$ holds for a.e. $s \in (\delta,q)$. As a consequence, it follows that
$$R-r\leq \int_{\overline{\Psi}_{R,\mathrm{rad}}(r)}^q \frac{1}{\sqrt{W_{\mathrm{rad}}}}\leq \sqrt{2} \, I_q, \forall r \in (l,R).$$
This proves that the function $\overline{\Psi}_R$ vanishes in the ball $\overline{B_{R-\sqrt{2}I_q}}$, provided that $R \geq(4n+\sqrt{2})I_q$.

Conversely, we are going to establish that when $n=1$, the existence of a dead core region for $\beta_R:=\underline{\Psi}_R$ implies that $I_q:=\int_0^q\frac{\dd s}{\sqrt{2 W_{\mathrm{rad}}(s)}}\leq R$. In view of Pohozaev identity \eqref{poz}, we have
\begin{align*}
\int_0^r \big(W_{\mathrm{rad}}(\beta_R)- \frac{1}{2} |\beta'_R|^2\big)=r \big(W_{\mathrm{rad}}(\beta_R(r))- \frac{1}{2} |\beta'_R(r)|^2\big) \text{ for a.e. } r \in (0,R),
\end{align*}
which implies that 
\begin{align}\label{ham}
 \frac{1}{2} |\beta'_R(r)|^2- W_{\mathrm{rad}}(\beta_R(r))=H \text{ for a.e. } r \in (0,R),
\end{align}
for some constant $H$\footnote{We point out that \eqref{ham} expresses the conservation of the total mechanical energy for the solutions of the Hamiltonian system $u''(x)=\nabla W(u(x))$. This property still holds for minimizers of \eqref{ene}, in the case of nonsmooth potentials}.  By assumption $\beta$ vanishes on a small interval $[0,\epsilon]$, thus it follows from \eqref{ham}, that actually
\begin{align}\label{ham2}
 \frac{1}{2} |\beta'_R(r)|^2= W_{\mathrm{rad}}(\beta_R(r)) \text{ for a.e. } r \in (0,R).
\end{align}
Let $(l,R]$ be the interval where $\beta_R>0$. Since $\beta_R$ is strictly increasing on the interval $(l,R)$, we denote its inverse function by $\gamma_R:(0,q)\to (l,R)$. As previously, we can see that $\gamma_R$ is locally Lipschitz on $(0,q)$, and that 
$\gamma'_R(s)=\frac{1}{\sqrt{2W_{\mathrm{rad}}(s)}}$ holds for a.e. $s \in (0,q)$. Therefore, we conclude that $I_q:=\int_0^q\frac{\dd s}{\sqrt{2 W_{\mathrm{rad}}(s)}}\leq R$.

So far we have proved that $\int_0^q\frac{\dd s}{\sqrt{2 W_{\mathrm{rad}}(s)}}=\infty$, implies that $\beta_R(r)>0$, for every $R>0$, and $r\in (0,R)$. Actually, the functions $\beta_R$ are positive on the whole interval $[-R,R]$. Indeed, in view of Lemma \ref{pl3a} (c), we have $\beta_R(r)\geq \overline{\Psi}_{2R}(2r)$, $\forall r \in [-R,R]$. Next, Theorem \ref{th1} applied in $\Omega=(-6R,2R)$, with $u(s)=\beta_{4R}(2R+s)$ and $\overline{\Psi}_{2R}(s)$, gives the inequality $\beta_{4R}(2R+s)\leq\overline{\Psi}_{2R}(s)$, $\forall s \in [-2R,2R]$, from which we deduce that $0<\beta_{4R}(2R+2r)\leq\beta_{R}(r)$, $\forall r \in (-R,R)$.
To complete the proof of Lemma \ref{pl3}, it remains to establish that the condition $\int_0^q\frac{\dd s}{\sqrt{2 W_{\mathrm{rad}}(s)}}=\infty$, also implies the positivity of the functions $\underline{\Psi}_R$, in higher dimensions $n \geq 2$. To see this, we  apply Theorem \ref{th1} in $\Omega=(-R,R)^n$, to the minimizer $u(x_1,x_2,\ldots,x_n)=\beta_R(x_1)$, and we get $0<\beta_R(r)\leq \overline{\Psi}_{R,\mathrm{rad}}(r)$, $\forall r\in [0,R]$. Finally, in view of Lemma \ref{pl3a} (c), we conclude that the functions $\underline{\Psi}_R$ are positive.
\end{proof}

\begin{remark}
In view of Lemma \ref{pl3a}, if $\overline{\Psi}_R$ or $\underline{\Psi}_R$ has a dead core, then $\overline{\Psi}_S$ and $\underline{\Psi}_S$ have also a dead core for every $S>R$. As a consequence, assuming that $I_q<\infty$, there exists a critical value $R_0$ such that $\overline{\Psi}_R$ has a dead core for $R>R_0$, while $\underline{\Psi}_R$ does not have a dead core for $R<R_0$. Lemma \ref{pl3} establishes that $R_0\leq(4n+\sqrt{2})I_q$ holds in any dimension $n$. On the other hand, in Lemma \ref{lemad} below, we determine the value of $R_0$, when $n=2$ and $W_{\mathrm{rad}}$ is the characteristic function of $\R\setminus \{0\}$. We refer to \cite[Section 8.4.]{pucci3} for the general theory of dead cores in the \emph{smooth} case, and in particular to \cite[Theorems 8.4.2., 8.4.3., 8.4.4.]{pucci3} for the properties of the function $\overline{\Psi}_R$. In \cite[Section 8.4.]{pucci3}, several explicit examples of dead cores are also provided.
\end{remark}

\begin{proof}[Proof of Proposition \ref{prop2}]
In view of the continuity of $u$ (cf. Lemma \ref{lem2}), we have $u\geq \epsilon$ on $\partial\omega$, for some $\epsilon>0$. Let $B_R$ be a ball containing the domain $\omega$. By increasing $R$, we may assume that the functional $J_{B_R}$ admits a unique minimizer $\Phi$ in the class $B:=\{h\in W^{1,2}(\Omega;\R): h=\epsilon \text{ on } \partial B_R\}$ (cf. Proposition \ref{prop1}, and Lemma \ref{pl3a}). 
In addition, it is clear that $(\Phi-u)^+,(u-\Phi)^-\in W^{1,2}_0(\omega)$. Thus, in view of the minimality of $\Phi$ and $u$, we have on the one hand
\begin{subequations}
\begin{equation}
E_{B_R}(\Phi)\leq E_{B_R}(\Phi-(u-\Phi)^- )\Leftrightarrow E_{\{\Phi>u\}}(\Phi)\leq E_{\{\Phi>u\}}(u),
\end{equation}
and on the other hand
\begin{equation}
E_{\omega}(u)\leq E_{\omega}(u+(\Phi-u)^+ )\Leftrightarrow E_{\{\Phi>u\}}(u)\leq E_{\{\Phi>u\}}(\Phi).
\end{equation}
\end{subequations}
That is, $ E_{\{\Phi>u\}}(\Phi)=E_{\{\Phi>u\}}(u)$, which means that $\Phi-(u-\Phi)^- $ is a minimizer of $J_{B_R}$ in the class $B$. By uniqueness of the minimizer $\Phi$, we conclude that $u\geq \Phi$ on $\omega$. In the case where $W_{\mathrm{rad}}(s)>0$, $\forall s \in (0,q]$, we have seen in Lemma \ref{pl3} that $\Phi>0$. Otherwise, if $W_{\mathrm{rad}}(\eta)=0$, for some $\eta\in (0,q)$, it is straightforward that $\Phi \geq \eta$ (cf. the proof of Lemma \ref{pl1}). 
\end{proof}

Finally, we close this section by detailing the computation mentioned in Remark \ref{rem1}:
\begin{lemma}\label{lemad}
Let $n=2$, and
\begin{equation*}
\tilde W_{\mathrm{rad}}(r)=\begin{cases}
0 &\text{if } r=0,\\
1 &\text{if } r>0.
\end{cases}
\end{equation*}
Then, when $R=R_0:=\sqrt{2e}\, q$, the functional $J_{B_R}$ admits exactly two radial minimizers in the class $A_R$, namely $\overline{\Psi}_R \equiv q$, and
\begin{equation*}
\underline{\Psi}_R(x)=\begin{cases}
0 &\text{if } |x|\leq \sqrt{2}\, q,\\
2q \ln (\frac{|x|}{\sqrt{2}\, q}) &\text{if } \sqrt{2}\, q\leq |x|\leq R_0.
\end{cases}
\end{equation*}
On the other hand, when $R>R_0$ (resp. $R<R_0$), $J_{B_R}$ admits only one radial minimizer in the class $A_R$, namely
\begin{equation*}
\psi_R(x)=\begin{cases}
0 &\text{if }  |x|\leq a_R,\\
q \frac{ \ln (|x|)-\ln a_R}{\ln R- \ln a_R} &\text{if }  |x|\geq a_R,
\end{cases}
\end{equation*}
where $a_R$ is the only solution of $\sqrt{2} \, a\ln(R/a)=q$ in the interval $(\frac{R}{\sqrt{e}},R)$ (resp. $\psi_R\equiv q$). 
\end{lemma}
\begin{proof}
We have seen in Lemma \ref{pl1} that the radial minimizers $\psi_R$ of $J_{B_R}$ in the class $A_{R}$, are such that the function $\psi_{R,\mathrm{rad}}$ is nondecreasing on the interval $[0,R]$. Moreover, due to our specific choice of $\tilde W_{\mathrm{rad}}$, these minimizers $\psi_R$ are harmonic functions on the set $\{x\in B_R:\psi_R(x)>0\}$. Thus, either $\psi_R\equiv q$ and $J_{B_R}(\psi_R)=\pi R^2$, if $\{x\in B_R:\psi_R(x)>0\}=B_R$, or 
\begin{equation*}
\psi_R(x)=\begin{cases}
0 &\text{if } |x|\leq a,\\
q\frac{ \ln (|x|)-\ln a}{\ln R- \ln a} &\text{if } |x|\geq a,
\end{cases}\text{ for some $a\in (0,R)$,}
\end{equation*}
\begin{equation*}
\text{and }J_{B_R}(\psi_R)=\pi(R^2 -a^2)+\frac{\pi q^2}{(\ln R- \ln a)},
\end{equation*}
if $\{x\in B_R:\psi_R(x)>0\}=\{x\in B_R: a<|x|<R\}$. Finally, by studying the variations of the functions $(0,R)\ni a \mapsto g_R(a):=-\pi a^2+\frac{\pi q^2}{(\ln R- \ln a)}$, one can show that $g_R$ vanishes or takes negative values on $(0,R)$, 
iff $R\geq R_0:=\sqrt{2e}\, q$. Moreover, when $R=R_0$, we have $g_R\geq 0$ on $(0,R$), and $g_R$ only vanishes for $a=\sqrt{2}\, q=\frac{R_0}{\sqrt{e}}$. Otherwise, if $R>R_0$, $g_R$ attains its negative minimum at the point $a_R\in(\frac{R}{\sqrt{e}},R)$, solving the equation $\sqrt{2} \, a\ln(R/a)=q$. We also point out that the function $R\mapsto \frac{R}{a_R}$ is decreasing from $(R_0,\infty)$ onto $( 1,\sqrt{e})$.
This completes the proof of Lemma \ref{lemad}.
\end{proof}

\section{Pohozaev identity and continuity for minimizers}\label{sec:sec4}
Pohozaev identity is commonly used for smooth solutions of semilinear elliptic systems (cf. for instance \cite[Remark 3.1]{book}). We prove below that the identity also holds for minimizers of \eqref{ene}. 
\begin{lemma}\label{lem1}
Let $W$ be a nonnegative, bounded and lower semicontinuous function defined on $\overline{B_q}$, and let $u\in W^{1,2}_{\mathrm{loc}}(\Omega;\R^m)$ be a map defined in the domain $\Omega\subset\R^n$, and satisfying \eqref{qbound} as well as  \eqref{enemin}. Then, given a ball $\overline{B_R(x_0)}\subset\Omega$, we have
\begin{equation}\label{poz}
\int_{B_r(x_0)}\Big(\frac{n-2}{2}|\nabla u|^2+n W(u)\Big)=r\int_{\partial B_r(x_0)}\Big(\frac{1}{2}|\nabla u|^2+ W(u)-\Big|\frac{\partial u}{\partial \nu}\Big|^2\Big), \text{ for a.e. } r\in (0,R),
\end{equation}
where $\nu$ stands for the outer nornal to the ball $B_r(x_0)$.
\end{lemma}
\begin{proof}
Without loss of generality, we may assume that $x_0=0$. Let $r\in (0,R)$, and $s \in (r, R)$ be fixed.  Given $x\in B_s\setminus \{0\}$, we write $x=t \sigma$, with $t:=|x|$, and $\sigma \in \SF^{n-1}$. Moreover, we set $|u_t(t,\sigma)|^2:=|\frac{\partial u}{\partial t}(t,\sigma)|^2$, and $|\nabla _\sigma u(t,\sigma)|^2=|\nabla u(t,\sigma)|^2- |u_t(t,\sigma)|^2$. Next, we consider in $B_s$, the comparison map
\begin{equation}
\tilde u(x)=\tilde u (t,\sigma)=\begin{cases}
u(\frac{x}{\kappa}) &\text{ when } 0\leq t=|x|\leq \kappa r,\\
u(r+(s-r)\frac{t-\kappa r}{s-\kappa r},\sigma) &\text{ when } \kappa r\leq t\leq s,
\end{cases}
\end{equation}
where $\kappa\in (0,\frac{s}{r})$ is fixed. It is clear that $\tilde u= u$ on $\partial B_s$, thus by the minimality of $u$, we have 
\begin{equation}\label{ineq22}
E_{B_s}(\tilde u)-E_{B_s}(u)\geq 0, \forall \kappa \in \big(0,\frac{s}{r}\big).
\end{equation}
Setting $f(\kappa):=E_{B_s}(\tilde u)$, a long but otherwise trivial computation shows that 
\begin{align*}
f(\kappa)&=\int_{B_r}\Big[\frac{\kappa^{n-2}}{2}|\nabla u|^2+\kappa^n W(u)\Big]\\
&+\int_r^s\int_{\SF^{n-1}}\big((s-\kappa r)\frac{t-r}{s-r}+\kappa r\big)^{n-1}\frac{s-r}{s-\kappa r}\frac{|u_t(t,\sigma)|^2}{2}\dd t \dd \sigma\\
&+\int_r^s\int_{\SF^{n-1}}\big((s-\kappa r)\frac{t-r}{s-r}+\kappa r\big)^{n-3} t^2\frac{s-\kappa r}{s- r}\frac{|\nabla_\sigma u_t(t,\sigma)|^2}{2}\dd t \dd \sigma \\
&+
\int_r^s\int_{\SF^{n-1}}\big((s-\kappa r)\frac{t-r}{s-r}+\kappa r\big)^{n-1}\frac{s-\kappa r}{s-r}W(u(t,\sigma))\dd t \dd \sigma,
\end{align*}
and
\begin{multline}\label{penult}
f'(1)=\int_{B_r}\Big[\frac{n-2}{2}|\nabla u|^2+n W(u)\Big]+
\frac{r}{s-r}\int_{B_s\setminus B_r}\big((n-1)\frac{s}{t}-(n-2)\big)\frac{|u_t(t,\sigma)|^2}{2}\\
+\frac{r}{s-r}\int_{B_s\setminus B_r}\big((n-3)\frac{s}{t}-(n-2)\big)\frac{|\nabla_\sigma u_t(t,\sigma)|^2}{2}
+\frac{r}{s-r}
\int_{B_s\setminus B_r}\big((n-1)\frac{s}{t}-n)\big)W(u(t,\sigma)).
\end{multline}
Therefore, in view of \eqref{ineq22}, and since $f(1)=E_{B_s}(u)$, we deduce that $f'(1)=0$. Finally, letting $s\to r$ in \eqref{penult}, we obtain for a.e. $r\in (0,R)$:
\begin{equation*}
0=\int_{B_r}\Big[\frac{n-2}{2}|\nabla u|^2+n W(u)\Big]+r
\int_{\partial B_r}\big(\frac{|u_t|^2}{2}-\frac{|\nabla_\sigma u|^2}{2}-W(u)\big).
\end{equation*}
\end{proof}
\begin{remark}
Proceeding as in \cite[page 91]{book}, one can also derive from Pohozaev identity the monotonicity formula $\frac{\dd }{\dd r} (r^{-(n-2)}E_{B_r(x_0)}(u))\geq 0$, holding for a.e. $r\in (0,R)$, under the assumptions of Lemma \ref{lem1}. We refer to the expository papers \cite{evans,lussardi} for a detailed account of monotonicity formulae.
\end{remark}

Next, we recall the continuity of bounded minimizers of \eqref{ene}. This property is crucial in the proof of Theorem \ref{th1}. 
\begin{lemma}\label{lem2}
Let $W$ be a nonnegative, bounded and lower semicontinuous function defined on $\overline{B_q}$, and let $u\in W^{1,2}_{\mathrm{loc}}(\Omega;\R^m)$ be a map satisfying \eqref{qbound} and \eqref{enemin}. Then, $u$ is continuous in $\Omega$.
\end{lemma}
\begin{proof}
We refer to \cite[Lemma 2.1]{agz} where a logarithmic estimate is established for the local minimizer $u$, implying in particular its H{\"o}lder continuity.
\end{proof}

\section*{Acknowledgments}
This research is supported by REA - Research Executive Agency - Marie Sk{\l}odowska-Curie Program (Individual Fellowship 2018) under Grant No. 832332, by the Basque Government through the BERC 2018-2021 program, by the Ministry of Science, Innovation and Universities: BCAM Severo Ochoa accreditation SEV-2017-0718, by project MTM2017-82184- R funded by (AEI/FEDER, UE) and acronym ``DESFLU''.


\begin{thebibliography}{10}


\bibitem{alikakos2}
Alikakos, N. D., Fusco, G.:
Density estimates for vector minimizers and applications.
Discrete and continuous dynamical systems \textbf{35} No.~12, 5631--5663 (2015), Special issue edited by E.Valdinoci


\bibitem{book}
Alikakos, N. D., Fusco, G., Smyrnelis, P.: Elliptic systems of phase transition type. Progress in Nonlinear Differential Equations and Their Applications \textbf{91}, Springer-Birkh{\"a}user (2018).




\bibitem{agz}
Alikakos, N. D., Gazoulis, D., Zarnescu, A.: Existence of entire solutions of the Allen-Cahn system possessing free boundaries. Preprint.

\bibitem{alt}
Alt, H. W., Caffarelli, L., Friedman, A.: Variational problems with two phases and their free boundaries. Trans. Amer. Math. Soc. \textbf{282}, 431--461 (1984)


\bibitem{aris}
Aris, R.: The Mathematical Theory of Diffusion and Reaction in Permeable Catalysts. Clarendon Press, Oxford (1975)





\bibitem{bandle}
Bandle C., Sperb R. and Stakgold I.: Diffusion and reaction with monotone
kinetics. J. Nonlinear Analysis \textbf{8}, 321--333 (1984)

\bibitem{benilan}
Benilan, P., Brezis, H., Crandall, M.: A semilinear equation in $L^1(\R^n)$); Ann. Scuola Norm. Sup. Pisa. Sci. \textbf{4} (1975) 523--555.



\bibitem{caf}
Caffarelli, L., C\'{o}rdoba, A.: Uniform convergence of a singular perturbation problem.
Comm. Pure Appl. Math. \textbf{48}, 1--12 (1995)

\bibitem{diaz}
D\'iaz J. I, Hern\'andez J.: On the existence of a free boundary for a class of reaction diffusion
systems. SIAM J. Math. Anal., \textbf{15} (4), 670--685.


\bibitem{dbr}
Dr\'abek, P., Robinson, S. B.: Continua of local minimizers in a non-smooth model of phase transitions. Z. Angew. Math. Phys. \textbf{62}, 609--622 (2011)

\bibitem{evans}
Evans, L. C.: Monotonicity formulae for variational problems, Phil. Trans. R. 
Soc. A \textbf{371}: 20120339.


\bibitem{fried}
Friedman, A., Phillips, D.: The free boundary of a semilinear elliptic equation. Tr. Amer. Math. Soc. \textbf{282}, No.~1 (1984), 153--182.



\bibitem{lussardi}
Fried, E., Lussardi, L.: Monotonicity formulae for smooth extremizers of integral functionals. Rendiconti Lincei - Matematica e Applicazioni \textbf{30}, No.~2 365--377 (2019)


\bibitem{pucci1}
Pucci, P., Serrin, J.: The strong maximum principle revisited. J. Differential Equations \textbf{196}, 1--66 (2004) 


\bibitem{pucci2}
Pucci, P., Serrin, J.: Dead cores and bursts for quasilinear singular elliptic equations. SIAM J. Math. Anal. \textbf{38}, No.~1 259--278 (2006) 
 
 
\bibitem{pucci3}
Pucci, P., Serrin, J.: The maximum principle. Progress in Nonlinear Differential Equations and Their Applications \textbf{73}, Springer-Birkh{\"a}user (2007)

\bibitem{savin}
Savin, O.: Minimal Surfaces and Minimizers of the Ginzburg Landau energy. Cont. Math. Mech. Analysis AMS \textbf{526}, 43--58 (2010)

\bibitem{smy}
Smyrnelis, P.: Connecting orbits in Hilbert spaces and applicatons to P.D.E. Comm. Pure Appl. Anal. \textbf{19}, No.~5    2797--2818 (May 2020)


\bibitem{sperb}
Sperb, R.: Some complementary estimates in the dead core problem. 
Nonlinear Problems in Applied Mathematics. In honor of Ivar
Stakgold on his 70th birthday, T. S. Angell, et al. (eds.), Philadelphia,
(1996) 217--224.

\bibitem{vaz}
V\'azquez, J.-L.: A strong maximum principle for some quasilinear elliptic equations, Appl. Math. Optim. \textbf{12} (1984) 191--202.

\end{thebibliography}
\end{document}